\documentclass[11pt]{amsart}
\usepackage[margin=1.5in]{geometry}
\usepackage{amsmath, amsthm,amssymb}
\usepackage{
	comment,
	enumitem, %
	float, %
	wrapfig,
	subcaption }
\usepackage{microtype}
\usepackage{natbib}
\usepackage{color}
\usepackage{enumitem}
\def\F{{\mathcal F}}
\def\R{{\mathbb R}}

\newtheorem{theorem}{Theorem}[section]

\newtheorem{corollary}[theorem]{Corollary}

\newtheorem{lemma}[theorem]{Lemma}
\theoremstyle{definition}

\newtheorem{definition}[theorem]{Definition}

\newtheorem{assumption}[theorem]{Assumption}
\newtheorem{remark}[theorem]{Remark}
\newtheorem{example}[theorem]{Example}

\renewcommand{\epsilon}{\varepsilon}

\numberwithin{equation}{section}
\title[Sensitivity of multiperiod optimization problems]{Sensitivity of multiperiod optimization problems with respect to the adapted Wasserstein distance}
\author{Daniel Bartl}
\address{\hspace{-0.53cm} Daniel Bartl \newline
Faculty of Mathematics\newline
University of Vienna, Austria\newline
daniel.bartl@univie.ac.at}
\author{Johannes Wiesel}
\address{\hspace{-0.53cm} Johannes Wiesel\newline 
Department of Statistics \newline
Columbia University \newline
johannes.wiesel@columbia.edu}
\date{\today}
 
\begin{document}

\begin{abstract}
We analyze the effect of small changes in the underlying probabilistic model on the value of multiperiod stochastic optimization problems and optimal stopping problems. 
We work in finite discrete time and measure these changes with the adapted Wasserstein distance. We prove explicit first-order approximations for both problems.
Expected utility maximization is discussed as a special case.
\end{abstract}

\thanks{
Both authors thank  Mathias Beiglb\"ock, Yifan Jiang, and Jan Ob{\l}{\'o}j for helpful discussions and two referees for a careful reading.
DB acknowledges support from  Austrian Science Fund (FWF) through projects ESP-31N and P34743.
JW acknowledges support from NSF grant DMS-2205534.}

\maketitle

\section{Introduction}
Consider a (real-valued) discrete-time stochastic process $X=(X_t)_{t=1}^T$ whose probabilistic behavior is governed by a reference model $P$.
Typically, such a model could describe the evolution of the stochastic process in an idealized probabilistic setting, as is customary in mathematical finance, or it could be derived from historical observations, as is a common assumption in statistics and machine learning. 
In both cases, one expects $P$ to merely \emph{approximate} the true but unknown model.
Consequently, an important question pertains to the effect that (small) misspecifications of $P$ have on quantities of interest in these areas. 
In this note, we analyze this question in two fundamental instances: \emph{optimal stopping} problems and convex \emph{multiperiod stochastic  optimization} problems. 
For simplicity we focus on the latter in this introduction:  consider
\[
v(P):=\inf_{a \text{ admissable control} } E_P[f(X,a(X))],
\]
where $f\colon \mathbb{R}^T\times\mathbb{R}^T\to \mathbb{R}$ is convex in the control variable (i.e.,\ its second argument). 
The admissible controls are the (uniformly bounded)  predictable functions  $a=(a_t)_{t=1}^T$, i.e., $a_t(X)$ only depends on $X_1,\dots,X_{t-1}$.
For concreteness, let us mention that \emph{utility maximization}---an essential problem in mathematical finance---falls into this framework, by setting
\[ f(x,a):= -U\Big( g(x) - \sum_{t=1}^T a_t (x_{t}-x_{t-1}) \Big),\]
where $U\colon \mathbb{R}\to\mathbb{R}$ is a concave utility function, $g\colon\mathbb{R}^T\to\mathbb{R}$ is a payoff function and $X_0\in\mathbb{R}$---see  Example \ref{ex:1} for more details.

\vspace{0.5em}
The question how changes of the model $P$ influence the value of $v(P)$ clearly depends on the chosen distance between models.
In order to answer it in a generic way (i.e., without restricting to parametric models), one first needs to choose a suitable metric on the laws of stochastic processes $\mathcal{P}(\mathbb{R}^T)$.
A crucial observation that has appeared in different contexts and goes back at least to \cite{Al81,Hellwig,Pflug:2010hl,Pflug:2012bf}, is that any metric compatible with weak convergence (and also variants thereof that account for integrability, such as the Wasserstein distance) is too coarse to guarantee continuity of the map $P\mapsto v(P)$ in general.
Roughly put, the reason is that two processes can have very similar laws but completely different filtrations; hence  completely different sets of admissible controls. 
This fact has been rediscovered several times during the last decades, and researchers from different fields have introduced modifications of the weak topology that guarantee such continuity properties of $P\mapsto v(P)$; we refer to \cite{backhoff2020all} for detailed historical references.
Strikingly, all the different modifications of the weak topology turn out to be equivalent to the so-called \emph{weak adapted topology}: this is the coarsest topology that makes multiperiod optimization problems continuous, see \cite{backhoff2020all, bartl2021wasserstein}.

With the choice of topology  settled, the next question pertains to the choice of a suitable distance.
This is already relevant in a one-period framework, where the weak and  weak adapted topologies coincide. Recent research shows, that the Wasserstein distance, which metrizes the weak topology, is (perhaps surprisingly) powerful and versatile.
Analogously, the \textit{adapted Wasserstein distance} $\mathcal{AW}_p$ (see Section \ref{sec:main} for the definition) metrizes the weak adapted topology and \cite{Pflug:2010hl, backhoff2020adapted} show, that the multiperiod optimization problems considered in this note are \emph{Lipschitz-continuous} w.r.t.\ $\mathcal{AW}_p$.
However, the Lipschitz-constants depend on global continuity parameters of $f$ and are thus far from being sharp in general.
Moreover, the exact computation of the (worst case) value of $v(Q)$, where $Q$ is in a neighbourhood of $P$, requires solving an infinite-dimensional optimization problem, which does not admit explicit solutions in general.
Both of these issues already occur in a one-period setting---despite the results of, e.g., \cite{blanchet2019quantifying, Bartl:2019hk}, which relate this  infinite-dimensional optimization problem to a simpler dual problem.
In conclusion, \emph{computing} the error $$\mathcal{E}(r):= \sup_{\mathcal{AW}_p(P,Q)\leq r}v(Q)-v(P)$$ exactly is only possible in a few (arguably degenerate) cases.\\

In this note we address both issues by extending ideas of \cite{bartl2021sensitivity,obloj2021distributionally} from a one-period setting to a multiperiod setting.
The key insight of these works is that in a one-period setting, computing \emph{first-order approximations} for $\mathcal{E}(r)$ is virtually always possible, while obtaining exact expressions  might be infeasible in many cases.
Our results go hand in hand with those of \cite{bartl2021sensitivity}, and we obtain explicit closed-form solutions for $\mathcal{E}'(0)$, which have intuitive interpretations.
For instance, we  show in Theorem \ref{thm:sensitivity_opt}, that under mild integrability and differentiability assumptions 
\begin{align*}
\sup_{Q \, : \, \mathcal{AW}_p(P,Q)\leq r} v(Q)
 &=  v(P)
  + r\cdot \Big( \sum_{t=1}^T E_P\big[ \big| E_P \big[ \partial_{x_t} f(X,a^\ast(X)) \big| \mathcal{F}_t^X\big]\big|^{\frac{p}{p-1}} \big] \Big)^{\frac{p-1}{p}}  +o(r)
\end{align*}
holds as $r\downarrow 0$, where $\partial_{x_t} f(x,a)$ is the partial derivative with respect to the $t$-th coordinate of $x$, $\mathcal{F}^X_t=\sigma(X_1,\dots,X_t)$, $a^\ast$ is the unique optimizer for $v(P)$ and  $o$ denotes the Landau symbol.
In the case of utility maximization with $p=q=2$ (and $g\equiv 0$ for simplicity), the first-order correction term is essentially the expected quadratic variation of $a^\ast$, but not under $P$, but distorted by the $l^2$-distance of the conditioned Radon-Nykodym density of an equivalent martingale measure w.r.t.\ $P$---see Example \ref{ex:1} for details.
\\

Investigating  robustness of optimization problems in varying formulations is a recurring theme in the optimization literature; we refer to \cite{rahimian2019distributionally} and the references therein for an overview. 
In the context of mathematical finance, representing distributional uncertainty through Wasserstein neighbourhoods  goes back (at least) to  \cite{pflug2007ambiguity} and has seen a spike in recent research activity, leading to many impressive developments, see, e.g., duality results in \cite{gao2016distributionally, blanchet2019quantifying, kuhn2019wasserstein, Bartl:2019hk} and applications in mathematical finance \cite{blanchet2021distributionally}, machine learning and statistics \cite{shafieezadeh2019regularization, blanchet2020machine}.
Our theoretical results are directly linked to \cite{acciaio2022convergence,backhoff2022estimating} characterizing the speed of convergence between the true and the (modified) empirical measure in the adapted Wasserstein distance and to new developments for computationally efficient relaxations of optimal transport problems, see \cite{ eckstein2022computational}.
For completeness, we mention that other notions of distance have been used to model distributional uncertainty, see e.g., \cite{lam2016robust, lam2018sensitivity, calafiore2007ambiguous} in the context of operations research, \cite{huber2011robust, lindsay1994efficiency} in the context of statistics, and  \cite{herrmann2017model, hobson1998volatility, karoui1998robustness} in the context of mathematical finance.

\section{Main results}\label{sec:main}

\subsection{Preliminaries}

We start by setting up notation.
Let $ T\in\mathbb{N}$, let $\mathbb{R}^T$ be the path space of a stochastic process in finite discrete time, and let $\mathcal{P}_p(\R^T)$ denote the set of all Borel-probability measures on $\R^T$ with finite $p$-th moment.
Throughout this article, $X\colon\R^T\to\R^T$ is the identity (i.e., the canonical process) and $X,Y\colon\R^T\times\R^T\to\R^T$ denote the projections to the first and second coordinate, respectively.
The filtration generated by $X$ is denoted by $(\mathcal{F}^X_t)_{t=0}^T$, i.e., $\mathcal{F}_t^X:=\sigma(X_s : s\leq t)$ and  $\F_0^X:=\{\emptyset, \R^T\}$.
Sometimes we write $(\mathcal{F}^{X,Y}_t)_{t=1}^T$ for the filtration generated by the processes $(X_t,Y_t)_{t=1}^T$.

For a function $f\colon\R^T\times\R^T\to\R$ we write $\partial_{x_t}f$ for the partial derivative of $f$ in $t$-th coordinate of $x$, that is, $$\partial_{x_t}f(x,a)=\lim_{\varepsilon\downarrow0} \frac{1}{\varepsilon}(f(x+\varepsilon e_t,a)-f(x,a))$$ where $e_t$ is the $t$-th unit vector;   $\nabla_a f$ for the gradient in $a$, and $\nabla_a^2 f$ for the Hessian in $a$.
We adopt the same notation for functions $f\colon\R^T\to\R$ or $f\colon\R^T\times\{1,\dots,T\}\to\R$ and write $\partial_{x_t}f$ for the partial derivative of $f$ in $t$-th coordinate of $x\in\R^T$.
For univariate functions $\ell:\R\to \R$ we simply write $\ell', \ell''$ for the first and second derivatives.

\begin{definition}%
	Let $P,Q\in\mathcal{P}_p(\R^T)$.
	A Borel probability measure $\pi$ on $\R^T\times\R^T$ is called  \emph{coupling} (between $P$ and $Q$), if its first marginal distribution is $P$ and its second one is $Q$.
	A coupling $\pi$ is called \emph{causal} if 
	\begin{align}
	\label{eq:causal}
	\pi( (Y_1,\dots, Y_t)\in A \, |\,  X_1,\dots,X_T)
	= \pi( (Y_1,\dots, Y_t)\in A \,  |\,  X_1,\dots,X_t)
	\end{align}
	$\pi$-almost surely for all Borel sets $A\subseteq \R^t$ and all $1\leq t\leq T$; a casual coupling is called \emph{bicausal} if \eqref{eq:causal} holds also with the roles of $X$ and $Y$ reversed.
\end{definition}

Phrased differently, \eqref{eq:causal} means that under $\pi$, conditionally on the `past' $X_1,\dots,X_t$, the `future' $X_{t+1},\dots,X_T$ is independent of $Y_1,\dots,Y_t$; see e.g.\ \cite[Lemma 2.2]{bartl2021wasserstein} for this and further equivalent characterizations of (bi-)causality.
It is also instructive to analyze condition \eqref{eq:causal} in the case of a Monge-coupling, i.e., when  there is a transport map  $\psi\colon\R^T\to\R^T$ such that $Y=\psi(X) =(\psi_t(X))_{t=1}^T$ $\pi$-almost surely. 
Indeed, then \eqref{eq:causal} simply means that $\psi_t$ needs to be $\mathcal{F}^X_t$-measurable.

Fix $p\in(1,\infty)$ and define the \emph{adapted Wasserstein distance} on $\mathcal{P}_p(\R^T)$ by
\begin{align}
\label{eq:def.AW}
\mathcal{AW}_p(P,Q):=\inf_{\pi} \Big( \sum_{t=1}^T E_\pi [   |X_t-Y_t|^p ]\Big)^{1/p}, 
\end{align}
where the infimum is taken over all bicausal couplings $\pi$ between $P$ and $Q$. %
Set
\[ B_r(P):=\{ Q \in\mathcal{P}_p(\R^T) :  \mathcal{AW}_p(P,Q)\le r\} \]
for $r\geq 0$, and denote by $q:=p/(p-1)$ the conjugate H\"older exponent of $p$. %

\subsection{The uncontrolled case}
We are now in a position to state the main results of the paper. We start with a simplified case, where $f$ depends on $X$ only and there are no controls. The sensitivities of the stochastic optimization and optimal stopping problems in Section \ref{sec:con} and \ref{sec:sto} respectively can be seen as natural extensions of this result; indeed the sensitivity computed in Theorem \ref{thm:main} already exhibits the structure, which is common to all our results presented here.

\begin{theorem}\label{thm:main}
	Let $f\colon\R^T\to\R$ be continuously differentiable and assume that there exists $c>0$ such that 
	\[\sum_{s=1}^T |\partial_{x_s} f(x)|
	\le c \Big(1+\sum_{s=1}^T |x_s|^{p-1} \Big)
	\]
	for every $x\in \R^T$.
	Then, as $r\downarrow 0$,
\begin{align*}
	\sup_{Q\in B_r(P)} E_Q[f(X)] 
	= E_P[f(X)]+ r\cdot \Big( \sum_{t=1}^T E_P\Big[ \big| E_P[ \partial_{x_t} f(X) | \mathcal{F}^X_t ] \big|^q \Big]\Big)^{1/q} + o(r).
\end{align*}
\end{theorem}

\subsection{Multiperiod stochastic optimization problems}\label{sec:con}
Fix a constant $L$ throughout this section, and denote by $\mathcal{A}$ the set of all  \emph{predictable controls} bounded by $L$, i.e., every $a=(a_t)_{t=1}^T\in\mathcal{A}$ is such that  $a_t\colon\mathbb{R}^T\to\mathbb{R}$ only depends on $x_1,\dots,x_{t-1}$ (with the convention that $a_1$ is deterministic) and that $|a_t|\leq L$ for a fixed constant $L$. 
Recall that
\[ v(Q)=\inf_{a\in\mathcal{A}} E_Q[ f(X, a(X))],\]
where $f\colon\mathbb{R}^T\times\mathbb{R}^T\to\mathbb{R}$ is  assumed to be convex in the control variable (i.e.,\ its  second argument).

\begin{assumption}
\label{ass:optim.f.strict.convex}
For every $x\in\mathbb{R}^T$,  $f(x,\cdot)$ is twice continuously differentiable and
strongly convex in the sense that $\nabla^2_a f(X,\cdot )\succ \varepsilon(X) I$ on $[- L,L]^T$  where $I$ is the identity matrix\footnote{For two $T\times T$-matrices $A$ and $B$, we write $A\succ B$ if $A-B$ is positive semidefinite, that is,  $\langle Az,z\rangle \geq \langle Bz,z\rangle$ for all $z\in\R^T$.} and $P(\varepsilon(X)>0)=1$.
Moreover, $f(\cdot, a)$ is differentiable for every $a\in \mathbb{R}^T$, its partial derivatives $\partial_{x_t} f$ are jointly continuous, and there is a constant $c>0$ such that
\begin{align*}
\sum_{s=1}^T |\partial_{x_s} f(x,a)|\le c\cdot \Big( 1+\sum_{s=1}^T |x_s|^{p-1} \Big)
\end{align*}
for every $x\in\mathbb{R}^T$  and $a\in[-L,L]^T$.
\end{assumption}

\begin{theorem}
\label{thm:sensitivity_opt}
	If Assumption \ref{ass:optim.f.strict.convex} holds true, then there exists exactly one $a^\ast\in\mathcal{A}$ such that $v(P)= E_P[f(X,a^\ast(X))]$. Furthermore, as $r\downarrow0$,
	\[\sup_{Q\in B_r(P)} v (Q) = v(P)  + r\cdot  \Big( \sum_{t=1}^T E_P\big[ |E_P[ \partial_{x_t} f(X,a^\ast(X)) | \mathcal{F}_t^X]|^q \big] \Big)^{1/q} + o(r).
	\]
\end{theorem}

\begin{remark}
	The restriction to controls that are uniformly bounded (i.e.,\ satisfy $|a_t(x)|\leq L$) is necessary to guarantee continuity of $Q\mapsto v(Q)$ in general.
	This can be seen easily in the utility maximization example below---even when restricting to models that satisfy a no-arbitrage condition, see, e.g., \cite[Remark 5.3]{backhoff2020adapted}.
\end{remark}

\begin{example}%
\label{ex:1}
	Let $\ell\colon\mathbb{R}\to\mathbb{R}$ be a convex loss function, i.e., $\ell$ is bounded from below and  convex.
	Moreover let $g\colon\mathbb{R}^T\to\mathbb{R}$ be (the negative of) a payoff function and consider the problem
	\[ u(P):=\inf_{a\in\mathcal{A}} E_P\Big[ \ell\Big( g(X) + \sum_{t=1}^T a_t(X)(X_t-X_{t-1}) \Big) \Big],\]
	where $X_0\in\mathbb{R}$ is a fixed value.
As discussed in the introduction, $u(P)$ corresponds to the utility maximisation problem with payoff $g$.
	
	Suppose that  $\ell$ is twice continuously differentiable with $|\ell'(u)|\leq c(1+|u|^{p-1})$ and $\ell''>0$, that $g$ is continuously differentiable with bounded derivative, and  that $P(X_{t-1}=X_t)=0$ for all $t$. 
	Then Assumption \ref{ass:optim.f.strict.convex} is satisfied.
\end{example}

The assumption that $P(X_{t+1}= X_t)=0$ is used to prove strong convexity in the sense of Assumption \ref{ass:optim.f.strict.convex}.
	In the present one-dimensional setting, it simply
	means that the stock price does not stay constant from time $t$ to $t+1$ with positive probability.
	It is satisfied, for instance, if $X$ is a Binomial tree under $P$, or if $P$ has a density with respect to the Lebesgue measure---in particular, if $X$ is a discretized SDE with non-zero volatility.
	Moreover the assumption that the derivative of $g$ is bounded can be relaxed at the price of restricting to $\ell$ with a slower growth.

\begin{corollary}
\label{cor:ut.max}
	In the setting of Example \ref{ex:1}:
	Let $a^\ast$ be the unique optimizer for $u(P)$, set $a^\ast_{T+1}:=0,$ $$Z:=g(X)+ \sum_{t=1}^T a^\ast_t(X_t-X_{t-1}),$$ and 	
	\begin{align*} 
	V:=\Big(\sum_{t=1}^T  E_P\Big[ \Big| (a^\ast_{t+1} - a^\ast_t) \cdot E_P\big[ \ell'( Z )|\mathcal{F}_t^X \big]  - E_P\big[ \ell'( Z ) \partial_{x_t} g(X) |\mathcal{F}_t^X \big] \Big|^q \Big]\Big)^{1/q}.
\end{align*}	
	Then 
	\begin{align*}
	\sup_{Q\in B_r(P)} u(Q)  =u(P) + r\cdot V + o(r)\qquad \text{as }r\downarrow 0.
	\end{align*}
\end{corollary}

Note that for $p=q=2$ and $g=0$, $F$ is essentially the expected quadratic variation of $a^\ast$, but not under $P$, but distorted by the $l^2$-distance of the conditioned Radon-Nykodym density of an equivalent martingale measure w.r.t.\ $P$.

\subsection{Optimal stopping problems}\label{sec:sto}
Let $f\colon\R^T\times\{1,\dots,T\} \to\mathbb{R}$ be such that $f(X,t)$ is  $\F_t^X$-measurable for $t=1, \dots, T$ and consider
\begin{align*}
 s(Q) := \inf_{\tau \in \mathrm{ST}}  \mathbb{E}_{Q}[ f(X,\tau)],
 \end{align*}
where $\mathrm{ST}$ refers to the set of all bounded stopping times with respect to the canonical filtration, i.e., $\tau\in \mathrm{ST}$ if $\tau\colon\R^T\to\{1,\dots,T\}$ is such that $\{\tau\leq t\}\in \mathcal{F}_t^X$ for every $t$.

\begin{theorem}\label{thm:opt_stopping}
	Assume that $f(\cdot,t)$ is continuously differentiable for every $t=1,\dots,T$ and that there is a constant $c>0$ such that 
	\[\sum_{s=1}^T | \partial_{x_s} f(x,t)|
	\le c \Big(1+\sum_{s=1}^T |x_s|^{p-1} \Big)
	\]
	for every $x\in\R^T$ and $t=1, \dots,T$.
	 Furthermore assume that there exists exactly one optimal stopping time $\tau^\ast$ for $s(P)$.
		Then, as $r\downarrow0$,	
	\[\sup_{Q\in B_r(P)} s(Q) = s(P) + r\cdot \left(\sum_{t=1}^T E_P\left[  \left|E_{P}\left[ \partial_{x_t} f(X,\tau^\ast) |\F_t^X\right]\right|^q\right]\right)^{1/q} +o(r). \]
\end{theorem}

\begin{example}
	It is instructive to consider Theorem \ref{thm:opt_stopping} in the special case where $f$ is Markovian, i.e., there is function $g\colon\R\to\R$ such that $f(x,t)=g(x_t)$ for all $t$ and $x$.
	Indeed, in this case, the first-order correction term  simplifies to $E_P[|g'(X_{\tau^\ast})|^q]^{1/q}$.
\end{example}

\subsection{Extensions and open questions}

To the best of our knowledge, this is the first work addressing the nonparametric sensitivity of multiperiod optimization problems (w.r.t.\ the adapted Wasserstein distance). 
Below we identify possible extensions, which are outside of the current scope of the current article. We plan to address these in future work.
\begin{enumerate}[wide, itemsep=0.2em,labelwidth=!, labelindent=0em]
\item
Our results extend to $\R^d$-valued stochastic processes and $\mathcal{AW}_p$ defined with an  arbitrary norm on $\R^d$. Similarly, the set of predictable controls in Theorem \ref{thm:sensitivity_opt} can be chosen to be any compact convex subset of $\R^d.$
The modifications needed are in line with \cite{bartl2021sensitivity}.
\item
A natural extension of our results from a financial perspective would be the analysis of sensitivities for robust option pricing:
let $P$ be a martingale law (i.e.\ $X$ is a $(\mathcal{F}_t^X)_{t=1}^T$-martingale under $P$) and  consider
\[ \sup_{Q\in B_r(P): \, Q \text{ is a martingale law} } E[f(X)] .\]
In one period models ($T=1$) this was carried out in \cite{bartl2021sensitivity,nendel2022}.
In a similar manner, it is natural to analyze sensitivity of robust American option pricing by considering only martingales in Theorem \ref{thm:opt_stopping}.

\item 
There are certain natural examples for $f$ that do not satisfy our regularity assumptions, e.g., %
in mathematical finance.
In a one-period framework, regularity of $f$ can be relaxed systematically, see  \cite{bartl2021sensitivity,  nendel2022}, and it is interesting to investigate if this is the case here as well.
\item
In some examples, the restriction to bounded controls is automatic, see, e.g., \cite{rasonyi2005utility}.
For instance, in the setting of Example \ref{ex:1} with $g=0$, we suspect that  similar arguments as used in \cite{rasonyi2005utility} might show that  a  ``conditional full support condition" of $P$ is sufficient to obtain  first-order approximation with unbounded strategies.
\item
We suspect that the assumption on the uniqueness of the optimizer in Theorems \ref{thm:sensitivity_opt} and \ref{thm:opt_stopping} can be relaxed.
Indeed, at least in a two-period setting, modifications of the arguments presented here can cover the general case in Theorem \ref{thm:opt_stopping}.

\item
Motivated from the literature on distributionally robust optimization problems cited in the introduction, one could also consider min-max problems of the form
\[
\inf_{a\in \mathcal{A}} \sup_{Q\in B_r(P)}  E_P[f(X,a(X))] .
\]
An important observation is that most arguments in the analysis of such problems (in the one-period setting) heavily rely on (convexity and) compactness of $B_r(P)$; both properties fail to hold true in multiple periods.
It was recently shown in \cite{bartl2021wasserstein} that these can be recovered by passing to an appropriate factor space of processes together with general filtrations.

\item
The present methods can be extended to cover functionals that depend not only on $P$ but also on its disintegrations---as is common in weak optimal transport (see, e.g., \cite{gozlan2017kantorovich}).
As an example, consider $T=2$ and $J(Q):= E_Q[ f(X_1)+g(Q_{X_1})]$, where the functions $f$ and $g$ are suitably (Fr\'echet) differentiable. 
Using the same arguments as in the proof of Theorem \ref{thm:main}, one can show that the first-order correction term equals $(E_P[ |f'(X_1)|^q+|E_P[g'(P_{X_1})]|^q])^{1/q}$.\footnote{When completing a first draft of this paper, we learned that similar results have been established by \cite{yifan} in independent research.}
\end{enumerate}

\section{Proofs}\label{sec:proofs}

\subsection{Proof of Theorem \ref{thm:main}}

We need the following technical lemma, which essentially states that causal couplings can be approximated by bicausal ones with similar marginals.
For a Borel probability measure $\pi$ on $\R^T\times\R^T$ and a Borel mapping $\phi$ from $\R^T\times\R^T$ to another Polish space, $\phi_\ast \pi$ denotes the push-forward of  measure $\pi$ under $\phi$.

\begin{lemma}
\label{lem:bi-causal}
Let $P,Q\in\mathcal{P}_p(\R^T)$ and let $\pi$ be a causal coupling between $P$ and $Q$. 
For each $\delta>0$ there exists $Y^\delta\colon \R^T\times\R^T\to\R^T$ such that $Y^\delta_t$ is $\F^{X,Y}_{t}$-measurable,  $X_t$ is $\sigma(Y^\delta_t)$-measurable, and $|Y_t^\delta-Y_t|\leq \delta$ for every $1\leq t\leq T$.

In particular, $\pi^\delta:=(X,Y^\delta)_\ast \pi$ is a bicausal coupling between $P$ and $Q^\delta:=Y^\delta_\ast\pi$.
\end{lemma}

\begin{proof}
	For $\delta>0$ we consider the Borel mappings
	\begin{align*}
	\psi_\delta &\colon\mathbb{R} \to (0,\delta) \text{ and}\\
	\phi_\delta &\colon\mathbb{R}\to \delta \mathbb{Z} :=\{ \delta k : k\in \mathbb{Z} \}  
	\end{align*}
	where $\psi_\delta$ is a (Borel-)isomorphism and $\phi_\delta(x):=\max\{ \delta k : \delta k \leq x\}$.
	For $t=1,\dots,T$, set
	\[ Y_t^\delta :=   \phi_\delta(Y_t) + \psi_\delta(X_t). \]
	By definition $X_t$ is $\sigma(Y_t^\delta)$-measurable, $Y^\delta_t$ is $\mathcal{F}^{X,Y}_{t}$-measurable, and $|Y_t^\delta-Y_t|\leq \delta$.
	It remains to note that the bicausality constraints \eqref{eq:causal} are clearly satisfied.
\end{proof}

The following proof serves as a blue print for the proofs of Theorems \ref{thm:sensitivity_opt} and \ref{thm:opt_stopping}.

\begin{proof}[Proof of Theorem \ref{thm:main}]
To simplify notation, set 
\[F_t:=E_P[ \partial_{x_t} f(X) | \mathcal{F}^X_t] \quad\text{for } t=1,\dots, T.\]
We first prove the \emph{upper bound}, that is
\begin{align}
\label{eq:simple.upper.bound}
\limsup_{r\to 0} \frac{1}{r} \Big( \sup_{Q\in B_r(P)} E_Q[f(X)] - E_P[f(X)] \Big)
\leq \Big( \sum_{t=1}^T E_P[ |F_t|^q ] \Big)^{1/q} .
\end{align}
To that end, for any $r>0$, let $Q=Q^r\in B_r(P)$ be such that 
\[E_Q[f(X)]\geq \sup_{R\in B_r(P)} E_R[f(X)] - o(r),\]
and let $\pi=\pi^r$ be an (almost) optimal bicausal coupling between $P$ and $Q$, i.e.,
\[ \Big( \sum_{t=1}^T E_{\pi}[ |X_t-Y_t|^p]\Big)^{1/p} 
\leq \mathcal{AW}_p(P,Q) + o(r)
\leq r + o(r).\]
The fundamental theorem of calculus and Fubini's theorem imply
\begin{align*}
E_Q[f(X)]-E_P[f(X)]
&=  E_\pi[ f(Y)-f(X)]   \\
&=\sum_{t=1}^T \int_0^1 E_\pi[ \partial_{x_t}f (X+\lambda (Y-X))\cdot (Y_t-X_t) ]\, d\lambda.
\end{align*}
Moreover, by the tower property and H\"older's inequality,
\begin{align*} 
&E_\pi[ \partial_{x_t}f (X+\lambda(Y-X)) \cdot (Y_t-X_t)]\\
&= E_\pi\Big[ E_\pi[\partial_{x_t}f (X+\lambda (Y-X))|\mathcal{F}_{t}^{X,Y}] \cdot (Y_t-X_t)\Big]\\
&\leq E_\pi\Big[ |E_\pi[\partial_{x_t}f (X+\lambda (Y-X))|\mathcal{F}_{t}^{X,Y}]|^q\Big]^{1/q}\cdot  E_\pi[|Y_t-X_t|^p]^{1/p}.
\end{align*}
We next claim that, for every $\lambda\in[0,1]$,
\begin{align}\label{eq:main_2}
E_\pi\Big[ \Big|E_\pi[\partial_{x_t}f (X+\lambda (Y-X) )|\mathcal{F}_{t}^{X,Y}]\Big|^q\Big]^{1/q}
\to E_P[ |F_t|^q]^{1/q}
\end{align}
as $r \to 0$. 
Indeed, since $\pi$ is bicausal, we have that 
\[ E_\pi[\partial_{x_t}f (X)|\mathcal{F}_{t}^{X,Y}]
= E_P[\partial_{x_t}f (X)|\mathcal{F}_{t}^{X}]
=F_t\]
$\pi$-almost surely, see, e.g., \cite[Lemma 2.2]{bartl2021wasserstein}.
Therefore, Jensen's inequality shows that
\begin{align*}
&E_\pi\Big[ \Big| E_\pi[\partial_{x_t}f (X+\lambda (Y-X) )|\mathcal{F}_{t}^{X,Y}] - F_t \Big|^q\Big]  \\
&\leq  E_\pi\Big[ | \partial_{x_t}f (X+\lambda (Y-X) ) - \partial_{x_t}f (X) |^q\Big]
\end{align*}
which converges zero; this follows from the continuity of $\partial_{x_t} f$, since $\sum_{t=1}^T E_\pi[ |X_t-Y_t|^p]\to 0$, and since $|\partial_{x_t} f(x)|^q \leq \tilde{c}(1+\sum_{s=1}^T |x_s|^p)$ by the growth assumption and since $q(p-1)=p$.
Then \eqref{eq:main_2}  follows from the triangle inequality.

We conclude that 
\begin{align*}
E_Q[f(X)] - E_P[f(X)]
&\leq \sum_{t=1}^T  \Big( E_P[ |F_t|^q]^{1/q} + o(1) \Big)  E_\pi[|Y_t-X_t|^p]^{1/p} \\
&\leq \Big( \sum_{t=1}^T  E_P[ |F_t|^q] +o(1)\Big)^{1/q} \Big( \sum_{t=1}^T  E_\pi[|Y_t-X_t|^p]\Big)^{1/p},
\end{align*}
where the second inequality follows from H\"older's inequality between $\ell^q(\mathbb{R}^T)$ and $\ell^p(\mathbb{R}^T)$.
Recalling that $\pi$ is an almost optimal bicausal coupling between $P$ and $Q$ and that $\mathcal{AW}_p(P,Q)\leq r$, this shows \eqref{eq:simple.upper.bound}.

\vspace{0.5em}
It remains to prove the \emph{lower bound}, that is, 
\begin{align}
\label{eq:simple.lower.bound}
\liminf_{r\to 0}  \frac{1}{r} \Big( \sup_{Q\in B_r(P)} E_Q[f(X)] - E_P[f(X)] \Big)
\geq \Big(\sum_{t=1}^T E_P[ |F_t|^q ] \Big)^{1/q}  .
\end{align}
To that end, we first use the duality between $\|\cdot\|_{\ell^q(\mathbb{R}^T)}$ and $\|\cdot\|_{\ell^p(\mathbb{R}^T)}$, which yields the existence of $a\in[0,\infty)^T$ satisfying
\[
\Big(\sum_{t=1}^T  E_P[ |F_t|^q]\Big)^{1/q} 
= \sum_{t=1}^T  E_P[|F_t|^q]^{1/q} a_t
\quad\text{ and } 
\sum_{t=1}^T a_t^p=1.
\]
Next we use duality between $\|\cdot\|_{L^q(P)}$ and $\|\cdot\|_{L^p(P)}$ which yields the existence of random variables $(Z_t)_{t=1}^T$ satisfying
\[ E_P[|F_t|^q]^{1/q} a_t
=E_P[F_t Z_t] \quad \text{and }
E_P[|Z_t|^p]^{1/p}=a_t
\]
for $t=1,\dots,T$.
Combining both results,
\begin{align}\label{eq:hoelder.lower}
\sum_{t=1}^T E_P[ F_t Z_t]
=\Big(\sum_{t=1}^T  E_P[ |F_t|^q]\Big)^{1/q} 
\quad\text{and}\quad
\sum_{t=1}^T E_P[|Z_t|^p] =1.
\end{align}
Note that, since $F_t$ is $\mathcal{F}^X_t$-measurable, $Z_t$ can be chosen $\mathcal{F}^X_t$-measurable as well.

At this point, for fixed $r>0$, we would like to define $Q^r$ as the law of $X+rZ$ and $\pi=\pi^r$ as the law of $(X,X+rZ)$.
Since $Z_t$ is $\mathcal{F}^X_t$-measurable, $\pi$ is clearly causal.
Unfortunately however, it does not need to be bicausal in general.
We thus first apply  Lemma \ref{lem:bi-causal} to  $P$ and $Q=Q^r$ with $\delta=1/n$, which yields measures $Q^n=Q^{r,n}$ and processes $Y^n=Y^{r,n}$ which satisfy the assertion of Lemma \ref{lem:bi-causal}.

Now fix $r>0$. 
Since 
\begin{align*}
\mathcal{AW}_p(P,Q^n)
&\leq \Big( \sum_{t=1}^T  E_{\pi}[ |X_t-Y_t^n|^p] \Big)^{1/p}\\
&\leq \Big( \sum_{t=1}^T (ra_t)^p  \Big)^{1/p} +\frac{T}{n}
=r +  \frac{T}{n},
\end{align*}
we have that, for every $\varepsilon>0$, 
\[\sup_{R\in B_{r+\varepsilon}(P)} E_R[f(X)] - E_P[f(X)]
\geq \lim_{n \to \infty} E_{\pi}[f(Y^n)-f(X) ]. \]
Using the fundamental theorem of calculus and Fubini's theorem as before, the fact that $Y^n_t$ is $\mathcal{F}^X_t$-measurable  shows that
\begin{align*}
E_{P}[f(Y^n)-f(X) ] 
&= \sum_{t=1}^T  \int_0^1 E_{P}\Big[  \partial_{x_t} f(X+\lambda (Y^n_t-X_t)) \cdot (Y^n_t-X_t) \Big]\,d\lambda\\
&=\sum_{t=1}^T  \int_0^1 E_{P}\Big[  E_{P}[ \partial_{x_t} f(X+\lambda (Y^n_t-X_t)) |\F_t^{X}] \cdot (Y^n_t-X_t)\Big]\,d\lambda\\
&\to \sum_{t=1}^T \int_0^1 E_{P}\Big[ E_P[  \partial_{x_t} f(X+\lambda rZ) |\F_t^{X}]\cdot rZ_t \Big]\,d\lambda,
\end{align*}
as $n\to\infty$, by the growth assumption since $Y^n_t-X_t\to Z_t$ in $L^p(P)$.

In a final step we let $r\to0$.
Applying the previous step to $\varepsilon=o(r)$ shows that
\begin{align*}
&\liminf_{r\to 0} \frac{ 1}{r}\Big( \sup_{R\in B_{r}(P)} E_R[f(X)] - E_P[f(X)] \Big) \\
&\geq \liminf_{r\to 0} \sum_{t=1}^T \int_0^1 E_{P}\big[ E_P [ \partial_{x_t} f(X+\lambda rZ) |\F_t^{X}]\cdot Z_t \big]\,d\lambda 
\\
&= \sum_{t=1}^T  E_{P}\big[ E_P[  \partial_{x_t} f(X) |\F_t^{X}]\cdot Z_t \big]\,d\lambda,
\end{align*}
where the equality holds by using the growth assumption on $|\partial_{x_t} f|$.
Recalling the choice of $(Z_t)_{t=1}^T$  (see \eqref{eq:hoelder.lower}) completes the proof.
\end{proof}

\subsection{Proof of Theorem \ref{thm:sensitivity_opt}}

The proof of Theorem \ref{thm:sensitivity_opt} has a similar  structure as the proof of Theorem \ref{thm:main}, but some additional arguments have to be made in order to take care of the optimization in $a\in \mathcal{A}$.
Throughout, we work under Assumption \ref{ass:optim.f.strict.convex}. We start with two auxiliary results.

\begin{lemma}
\label{lem:optim.continuity}
	Let $(Q_n)_{n\in \mathbb{N}}$ be such that $\mathcal{AW}_p(P,Q_n)\to 0$ for $n\to \infty$.
	Then $v(Q^n)\to v(P)$.
\end{lemma}
\begin{proof}
	Let $Q\in\mathcal{P}_p(\R^T)$ and let $\pi$ be a bicausal coupling between $P$ and $Q$.
	Let $\varepsilon>0$ be arbitrary and fix $a\in\mathcal{A}$ that satisfies $E_P[f(X,a(X))]\leq v(P)+\varepsilon$.
	Next define $b$ by 
	\[b_t:=E_\pi[a_t(X) | \F_T^Y] \qquad\text{for } t=1,\dots,T.\]
	By bicausality, $b_t$ is actually measurable with respect to $\F_{t-1}^Y$ (see, e.g., \cite[Lemma 2.2]{bartl2021wasserstein}) and clearly $|b_t|\leq L$; thus $b\in\mathcal{A}$.
	Moreover, convexity of $f(x,\cdot)$ implies that
	\begin{align*}
	 v(Q)
	&\leq E_Q[f(Y,b(Y))] 
	=E_\pi[ f(Y,E_\pi[a(X) | \F_T^Y])] 
	\leq E_\pi[ f(Y,a(X))].  
	\end{align*}
	The fundamental theorem of calculus and H\"older's inequality yield
	\begin{align*}
	 &E_\pi[ f(Y,a(X))] - E_P[   f(X,a(X))] \\
	&= \int_0^1 \sum_{t=1}^TE_\pi\Big[   \partial_{x_t}  f(X+\lambda (Y-X), a(X) )  (Y_t-X_t)  \Big]\, d\lambda\\
	&\le \int_0^1 \sum_{t=1}^T E_{\pi}[|\partial_{x_t}  f(X+\lambda (Y-X), a(X) )|^q]^{1/q} E_{\pi} [|Y_t-X_t|^p]^{1/p} \, d\lambda.
	\end{align*}
	Using the growth assumption and arguing as in the proof of Theorem \ref{thm:main}, the last term is at most of order $\mathcal{AW}_p(P,Q)$.
	As $\varepsilon$ was arbitrary, this shows  $v(Q) -v(P) \leq  O(\mathcal{AW}_p(P,Q))$ (where again $O$ denotes the Landau symbol) and reversing the roles of $P$ and $Q$ completes the proof.
\end{proof}

\begin{lemma}
\label{lem:optim.existence.unique}
	There exists exactly one $a^\ast\in\mathcal{A}$ such that $v(P)=E_P[f(X,a^\ast(X))]$.
\end{lemma}
\begin{proof}
	This is a standard result.
	The existence follows from Komlos' lemma \cite{komlos1967generalization} and uniqueness from strict convexity.	
\end{proof}

\begin{proof}[Proof of Theorem \ref{thm:sensitivity_opt}]
Let $a^\ast\in\mathcal{A}$ be the unique optimizer of $v(P)$ (see Lemma \ref{lem:optim.existence.unique}) and, for shorthand notation, set $ F_t:=  E_P[ \partial_{x_t} f(X,a^\ast(X)) | \mathcal{F}_t^X]$ for $t=1,\dots,T$.

We first prove the \emph{upper bound}.
We claim that it follows from combining the reasoning in the proof of Theorem \ref{thm:main} and Lemma \ref{lem:optim.continuity}.
Indeed, let $Q^r\in B_r(P)$ be such that $$v(Q^r) \geq \sup_{Q\in B_r(P)} v (Q) -o(r)$$ and let $\pi=\pi^r$ be a bicausal coupling between $P$ and $Q^r$ that is (almost) optimal for $\mathcal{AW}_p(P,Q^r)$.
Define $b^r\in\mathcal{A}$ by $b^r:=E_{\pi}[a^\ast(X) |\F_T^Y]$ and use convexity of $f(x,\cdot)$ to conclude that
\[ v(Q^r)\leq E_\pi[ f(Y,b^r(Y)) ]
\leq E_\pi[f(Y,a^\ast(X))]. \]
From here on, it follows from the fundamental theorem of calculus and H\"older's inequality just as in the proof of Theorem \ref{thm:main} that 
\begin{align*}
v(Q^r) - v(P) 
&\leq \sum_{t=1}^T E_\pi[\partial_{x_t} f(X,a^\ast(X)) \cdot (Y_t-X_t)] + o(r) \\
&\leq r \Big( \sum_{t=1}^T E_\pi[ |F_t|^q] \Big)^{1/q} +o(r).
\end{align*}
This completes the proof of the upper bound.

\vspace{0.5em}
We proceed with the \emph{lower bound}.
To that end, we start with the same construction as in the proof of Theorem \ref{thm:main}: 
let $\pi=\pi^r$ be the law of $(X,X+rZ)$ where $Z$ satisfies \eqref{eq:hoelder.lower}, that is, $Z_t$ is $\mathcal{F}_t^X$-measurable for every $t$ such that $$\sum_{t=1}^T E_P[|Z_t|^p]\leq 1\quad \text{and}\quad   \Big( \sum_{t=1}^T E_P[ |F_t|^q] \Big)^{1/q}
=\sum_{t=1}^T E_P[F_t Z_t] .$$
Again, $\pi$ might  be only  causal and not bicausal, and we need to rely on Lemma \ref{lem:bi-causal}. 
For the sake of a clearer presentation, we ignore this step this time.

For each $r$, let $a^r\in\mathcal{A}$ be almost optimal for $v(Q^r)$, that is $$E_{Q^r}[f(Y,a^r(Y))]\leq v(Q^r) +o(r).$$
Observe that,  by construction of $Q^r$ (i.e.,\ since $Z_t$ is $\mathcal{F}_t^X$-measurable for each $t$), there is  $b^r\in\mathcal{A}$ such that $\pi(a^r(Y)= b^r(X))=1$.

	Now let $(r_n)_n$ be an arbitrary sequence that converges to zero.
By Lemma \ref{lem:optim.stategies.converge} below,  after passing to subsequence $(r_{n_k})_k$, $b^{r_{n_k}}(X)$ converges to $a^\ast(X)$ $P$-almost surely.
Since $$v(P)\leq E_P[f(X,b^{r_{n_k}}(X))]$$ for all $k$, the fundamental theorem of calculus and the growth assumption imply
	\begin{align*}
	v(Q^{r_{n_k}}) -v(P)
	&\geq E_P[f(X+r_{n_k} Z, b^{r_{n_k}}(X) ) - f(X, b^{r_{n_k}} (X)) ] - o(r_{n_k}) \\
	&\geq \sum_{t=1}^T E_P[  E_P[ \partial_{x_t} f(X, b^{r_{n_k}}(X)) |\mathcal{F}_t^X ] \cdot  r_{n_k} Z_t ] - o(r_{n_k}).
	\end{align*}
	Since $b^{r_{n_k}}(X)\to a^\ast(X)$ $P$-almost surely, the continuity of $\partial_{x_t} f$ and the growth assumption imply that
	\begin{align*}
	\liminf_{k\to \infty} \frac{v(Q^{r_{n_k}}) -v(P)}{ r_{n_k} }
	&\geq\sum_{t=1}^T E_P[    E_P[\partial_{x_t} f(X, a^\ast(X)) |\mathcal{F}_t^X ]  \cdot Z_t  ].
	\end{align*}
	To complete the proof, it remains to recall the choice of $Z$ and that $(r_n)_n$ was an arbitrary sequence.
\end{proof}

\begin{lemma}
\label{lem:optim.stategies.converge}
	In the setting of the proof of Theorem \ref{thm:sensitivity_opt}:
	there exists  a subsequence $(r_{n_k})_k$ such that $b^{r_{n_k}}(X)\to a^\ast(X)$ $P$-almost surely.	
\end{lemma}
\begin{proof}
	Recall that $b^{r_n}$ was chosen almost optimally for $v(Q^{r_n})$, hence
	\begin{align*}
	v(Q^{r_n})  
	&\geq E_{P}[ f(X+r_nZ, b^{r_n}(X))] -o(r_n)
	\geq  E_P[ f(X, b^{r_n}(X)) ] -O(r_n),
	\end{align*}
	where the last inequality holds by continuity and the growth assumptions on $f$, see the proof of Theorem \ref{thm:sensitivity_opt}.
	Next recall that $\nabla_a^2 f(X,a)\succ \varepsilon(X) I$  for $a\in [-L,L]^T$ with $P(\varepsilon(X)>0)=1$.
	In particular, a second order Taylor expansion shows that
	\begin{align*}
	&E_P[ f(X, b^{r_n}(X)) ] - E_P[f(X,a^\ast(X))]\\
	&\geq E_P\Big[ \langle \nabla_a f(X,a^\ast(X)), b^{r_n}(X)-a^\ast(X) \rangle \Big]+  E_P\Big[\frac{ \varepsilon(X) }{2} \|b^{r_n}(X)-a^\ast(X)\|_{\ell^2(\R^T)}^2 \Big].
	\end{align*}
	The first term is non-negative by  optimality of $a^\ast$.
	Thus, since $v(Q^{r_n})\to v(P)$ by Lemma \ref{lem:optim.continuity}, this implies that  the second term must converge to zero.
	As $\varepsilon$ is strictly positive, this can only happen if $b^{r_n}(X)\to a^\ast(X)$ in $P$-probability.
	Hence, after passing to a subsequence, $b^{r_n}(X)\to a^\ast(X)$ $P$-almost surely.	
\end{proof}

\subsection{Proof of Corollary \ref{cor:ut.max}}
	For shorthand notation, set $$(a\cdot x)_T:=\sum_{t=1}^T a_t(x_t-x_{t-1}).$$
	The goal is to apply Theorem \ref{thm:sensitivity_opt} to the function
	\[ f(x,a):= \ell( g(x) + (a\cdot x)_T )\]
	for $(x,a)\in\R^T\times\mathbb{R}^T$.
	To that end, we start by checking Assumption \ref{ass:optim.f.strict.convex}.
	Since $g$ continuously differentiable and $\ell$ is twice continuously differentiable, the parts of Assumption \ref{ass:optim.f.strict.convex} pertaining  to the differentiability of $f$ hold true.
	Moreover, %
	\[\langle \nabla^2_a f(x,a) u,u\rangle 
	= \ell''( g(x) + (a\cdot x)_T ) \sum_{t=1}^{T}  u_t^2 (x_{t}-x_{t-1})^2\]
	for any $u\in\mathbb{R}^T$.
	Since $\ell''>0$ and $P(X_t=X_{t-1})=0$ for every $t$ by assumption, one can readily verify that there is $\varepsilon(X)$ with $P(\varepsilon(X)>0)=1$ such that $$\nabla^2_a f(X,\cdot )\succ \varepsilon(X) I\qquad\text{on }[- L,L]^T.$$
	Next observe that 
	\[\partial_{x_t} f(x,a)=\ell'( g(x) + (a\cdot x)_T) ( \partial_{x_t} g(x) + (a_t -a_{t+1})).\]
	A quick computation involving the growth assumption on $g'$ and $\ell'$ shows that 
	\[ |\partial_{x_t} f(x,a)| \leq c\Big(1+ \sum_{s=1}^T |x_s|^{p-1}\Big)\qquad\text{ for all }x\in \mathbb{R}^T \text{ and } a\in [-L,L]^T.\]
	In particular, Assumption \ref{ass:optim.f.strict.convex} is satisfied, and the proof follows by applying Theorem \ref{thm:sensitivity_opt}.
\qed

\subsection{Proof of Theorem \ref{thm:opt_stopping}}

We start with the \emph{upper bound}.
To that end, let $\tau^\ast$ be the optimal stopping time for $s(P)$, let $Q\in B_r(P)$ be such that $s(Q)\geq \sup_{R\in B_r(P)} s(R) - o(r)$, and let $\pi$ be a (almost) optimal bicausal coupling for $\mathcal{AW}_p(P,Q)$.
Using a similar argument as in Lemma \ref{lem:optim.continuity}, we can use the coupling $\pi$ to build a stopping time $\tau$ such that $$E_{Q}[f(X, \tau(X))]\leq E_\pi[f(Y,\tau^\ast(X))]$$---see \cite[Lemma 7.1]{backhoff2020all} or \cite[Proposition 5.8]{bartl2021wasserstein} for detailed proofs.
Under the growth assumption on $f$, the fundamental theorem of calculus and Fubini's theorem yield
\begin{align*}
s(Q) - s(P)
&\leq E_{\pi}[f(Y,\tau^\ast(X)) - f(X,\tau^\ast(X))]+o(r) \\
&= \int_0^1 \sum_{t=1}^T E_{\pi}\left[\partial_{x_t} f(X +\lambda (Y - X) ,\tau^\ast(X))\cdot (Y_t - X_t) \right] \,d\lambda+ o(r) \\
&\leq r \int_0^1 \Big(\sum_{t=1}^T E_{\pi} \left[ |E_\pi[\partial_{x_t} f (X +\lambda (Y - X),\tau^\ast(X))|\F_t^{X,Y}] |^q \right]\Big)^{1/q}  \,d\lambda\\
&\quad + o(r),
\end{align*}
where the last inequality follows from H\"older's inequality and since $$\sum_{t=1}^T E_\pi[ |X_t-Y_t|^p ] \leq r^p$$ in the same way as in the proof of Theorem \ref{thm:main}.
We also conclude using similar arguments that 
\begin{align*}
\lim_{r\to 0} E_{\pi}\left[ |E_\pi[\partial_{x_t} f(X+\lambda (Y - X), \tau^\ast(X))|\F_t^{X,Y}] |^q\right]
=  E_P \left[ | E_P[\partial_{x_t} f (X, \tau^\ast)|\F_t^X]|^q\right]
\end{align*}
for every $\lambda\in[0,1]$ and every $t=1,\dots, T$.
	
\vspace{0.5em}
We proceed with the \emph{lower bound}.
To make the presentation concise, we assume here that $T=2$---the general case follows from a (somewhat tedious) adaptation of the arguments presented here.
The assumption that the optimal stopping time $\tau^\ast$ is unique implies, by the Snell envelope theorem, that $$P(f(X,1)\neq E_P[f(X,2)|\F_{1}^X])=1;$$ in particular
\begin{align}
\label{eq:property.tau.ast}
\begin{split}
\{\tau^\ast=1\}&=\{ f(X,1)<E_P[f(X,2)|\F_{1}^X]\},
\\ 
\{\tau^\ast =2\}&=\{f(X,1)>E_P[f(X,2)|\F_{1}^X]\}.
\end{split}
\end{align}
As before, set $F_t:=E_P[\partial_{x_t} f(X, \tau^\ast)|\F_t^X]$ and take $Z$ that satisfies \eqref{eq:hoelder.lower}, i.e., $Z_t$ is $\mathcal{F}_t^X$-measurable for every $t$, and 
\begin{align}
\label{eq:def.Z}
\begin{split}
 E_P[|Z_1|^p]+E_P[|Z_2|^p]&\leq 1 \quad\text{and}\\
E_P[F_1 Z_1] + E_P[F_2 Z_2] &=( E_P[ |F_1|^q]  +E_P[ |F_2|^q] )^{1/q} .
\end{split}
\end{align}
Next, for every $r>0$, set
\begin{align*}
A^r&:=\{  f(X + r Z,1)<E_P[f(X+rZ,2)|\F_{1}^X] \}\cap \{\tau^\ast=1\},\\
B^r&:=\{  f(X+rZ,1 )>E_P[f(X+ rZ,2)|\F_{1}^X] \}\cap \{\tau^\ast=2\}.
\end{align*}
Define the process 
$$X^r:=X+rZ \mathbf{1}_{A^r\cup B^r}.$$
Since $A^r,B^r$ and $Z_1$ are $\F_1^X$-measurable,  the coupling $\pi^r:=(X,X^r)_\ast P$ is causal between $P$ and $P^r:= (X^r)_\ast P$.
Using Lemma \ref{lem:bi-causal} (just as in the proof of Theorem \ref{thm:main}), we can actually assume without loss of generality that $\pi^r$ is in fact bicausal and that $\mathcal{F}^{X^r}_t=\mathcal{F}^X_t$ for each $t$---we will leave this detail to the reader and proceed.

In particular, since $$|X_t-X_t^r|\leq r |Z_t|,$$ it follows from \eqref{eq:def.Z} that $\mathcal{AW}_p(P,P^r)\leq r$; thus
\begin{align}
\label{eq:snell.for.P.r}
 \sup_{Q\in B_r(P)} s(Q)
 &\geq \inf_{\tau\in\mathrm{ST}} E_P[ f(X^r,\tau(X^r))]\\
&=E_P[ f(X^r,1)\wedge E_P[f(X^r,2) | \mathcal{F}^X_1]],
\end{align}
where the equality holds by the Snell envelope theorem and since $\mathcal{F}^X_1=\mathcal{F}^{X^r}_1$.

Next note that  
\begin{align*}
f(X^r,1)< E[f(X^r,2)|\mathcal{F}^X_1] 
&\text{ and } f(X,1)< E[f(X,2)|\mathcal{F}^X_1] 
\qquad\text{on }A^r, \\
f(X^r,1)> E[f(X^r,2)|\mathcal{F}^X_1] 
&\text{ and } f(X,1)> E[f(X,2)|\mathcal{F}^X_1] 
\qquad\text{on }B^r, \\
f(X^r,1)\wedge E_P[f(X^r,2) | \mathcal{F}^X_1]
&=f(X,1)\wedge E_P[f(X,2) | \mathcal{F}^X_1]
\qquad\,\,\,\text{on }(A^r\cup B^r)^c\in\mathcal{F}_1^X.
\end{align*}
Combined with \eqref{eq:snell.for.P.r} and since
$$s(P)= E_P [f(X,1)\wedge E [f(X,2)|\F_{1}^X]],$$ we get
\begin{align*}
\sup_{Q\in B_r(P)} s(Q)- s(P)
&\geq E_{P}\big[ (f(X^r,1) -f(X,1)) \mathbf{1}_{A^r} \\
&\qquad+ (E_{P}[f(X^r,2)-f(X,2)|\F_{1}^X])\mathbf{1}_{B^r}] \\
&=E_{P}\big[ (f(X^r,1) -f(X,1)) \mathbf{1}_{A^r} + (f(X^r,2)-f(X,2))\mathbf{1}_{B^r}],
\end{align*}
where the  inequality holds by the tower property.
Using the fundamental theorem of calculus  just as in the proof of Theorem \ref{thm:main} shows that
\begin{align*}
\sup_{Q\in B_r(P)} \frac{1}{r} \big( s(Q) - s(P) \big)
&\geq  \sum_{t=1}^2 E_{P}\big[ \partial_{x_t} f(X,1)Z_t  \mathbf{1}_{A^r} + \partial_{x_t} f(X,2) Z_t  \mathbf{1}_{B^r}\big] - o(1)  \\
&\to  \sum_{t=1}^2 E_{P}\big[ \partial_{x_t} f(X,\tau^\ast) Z_t \big]
\end{align*}
as $r\downarrow0$, where the convergence holds because, by \eqref{eq:property.tau.ast},   $\mathbf{1}_{A_1^r}\to  \mathbf{1}_{\{\tau^\ast=1\}} $ and $\mathbf{1}_{A_2^r}\to \mathbf{1}_{\{\tau^\ast=2\}}$.
To complete the proof, it remains to recall the definition of $Z$, see \eqref{eq:def.Z}.
\qed

\bibliographystyle{plainnat}
\bibliography{paperslib}

\end{document}